\pdfoutput=1
 \documentclass[draftcls,onecolumn]{IEEEtran}

\IEEEoverridecommandlockouts                              

\usepackage{url}
\usepackage{bbm}
\usepackage{graphics} 
\usepackage{epstopdf}
\usepackage{times} 
\usepackage{amsmath} 
\usepackage{amssymb}  
\usepackage{multicol}
\usepackage{multirow}
\usepackage{color}
\usepackage{xparse}
\usepackage{bm}
\usepackage{cite}
\usepackage{pifont}
\usepackage{comment}
\usepackage[papersize={8.5in,11in}, left=0.75in, right=0.75in, top=1in, bottom=0.82in]{geometry}
\usepackage{lipsum}
\usepackage{amsthm}
\newtheorem{theorem}{Theorem}
\newtheorem*{theorem*}{Theorem}
\newtheorem{lemma}{Lemma}

\newtheorem*{lemma*}{Lemma}
\newtheorem*{remark*}{Remark}
\newtheorem*{fact*}{Fact}
\newtheorem*{proof*}{Proof}

\newtheorem{definition}{Definition}
\newtheorem{corollary}{Corollary}

\usepackage{algorithm}
\usepackage{algpseudocode}

\usepackage{mathtools}

\makeatletter
\newcommand{\algmargin}{\the\ALG@thistlm}
\makeatother
\newlength{\whilewidth}
\algdef{SE}[parWHILE]{parWhile}{EndparWhile}[1]
{\parbox[t]{\dimexpr\linewidth-\algmargin}{%
		\hangindent\whilewidth\strut\algorithmicwhile\ #1\ \algorithmicdo\strut}}{\algorithmicend\ \algorithmicwhile}%
\algnewcommand{\parState}[1]{\State%
	\parbox[t]{\dimexpr\linewidth-\algmargin}{\strut #1\strut}}

\newlength\myindent 
\begin{document}

	\title{\LARGE \bf
		Distribution System Voltage Control under Uncertainties using Tractable Chance Constraints
	}

	\author{Pan Li, Baihong Jin, Dai Wang and Baosen Zhang
		
		\thanks{Pan Li is with Facebook, Inc., Menlo Park, CA  94025. This work was done while she was with the Electrical Engineering Department, University of Washington, Seattle, WA 98195, {\tt\small {lipan419}@gmail.com}.}%
		
		\thanks{Baihong Jin is with the Electrical Engineering and Computer Science Department, University of California, Berkeley, CA, 94720, {\tt\small {bjin}@berkeley.edu}.}
		\thanks{Dai Wang is with Tesla Inc., Palo Alto, CA, 94304, {\tt\small {daiwang}@tesla.com}.}
		\thanks{Baosen Zhang is with the Electrical Engineering Department, University of Washington, Seattle, WA 98195,       {\tt\small { zhangbao}@uw.edu}.}%
	}

	\maketitle
	\thispagestyle{empty}
	\pagestyle{empty}
		\vspace{-0.3in}
	\begin{abstract}
		Voltage control plays an important role in the operation of electricity distribution networks, especially with high penetration of distributed energy resources. These resources introduce significant and fast varying uncertainties. In this paper, we focus on reactive power compensation to control voltage in the presence of uncertainties. We adopt a chance constraint approach that accounts for arbitrary correlations between renewable resources at each of the buses. We show how the problem can be solved efficiently using historical samples \textcolor{black}{analogously to the stochastic quasi-gradient methods}. We also show that this optimization problem is convex for a wide variety of probabilistic distributions. Compared to conventional per-bus chance constraints, our formulation is more robust to uncertainty and more computationally tractable. We illustrate the results using standard IEEE distribution test feeders.
	\end{abstract}



\section{Introduction}
Voltage control is crucial to stable operations of power distribution systems, where it is used to maintain acceptable voltages at all buses under different operating conditions \cite{LietAl2014}. To control voltage, reactive power is traditionally regulated through tap-changing transformers and switched capacitors \cite{ZhangetAl2013}. With recent advances in cyber-infrastructure for communication and control, it is also possible to utilize distributed energy resources (DERs, i.e., electric vehicles \cite{WangEtAl2016}, PV panels \cite{KanchevEtAl2011,Zhang2015})
to provide voltage regulation. There exists an extensive literature in controlling voltage in a distribution network, some of them focus on centralized control \cite{FarivarEtAl2012pes,ValvEtAl2013}, while the others address distributed algorithm \cite{ZhuEtAl2016,LietAl2014, ZhangetAl2013,SulcEtAl2016}. 

In this paper, we focus the problem of \emph{voltage control under uncertainties}. In addition to voltage control capabilities, DERs bring significant uncertainty and fast variations to the distribution system~\cite{TuritsynetAl2011,Zhang2015,Robbins13}. Since most distribution systems do not yet have real-time communication capabilities, a decision made must be valid for a set of possible conditions. For example, the solar PVs in a distribution system may communicate with the feeder (or some other controller) every 5 minutes to receive a command for setting its reactive power, but the changes in solar radiation result in sub-minute timescales changes in its active power. In this paper, we overcome these fast variations using a centralized control framework, where a central controller sends out regulation signals to DERs periodically, where the control signals are designed to regulate voltages for the entirety of the period in the presence of randomness.


A natural framework to handle the uncertainties introduced by the fast variation in the output of the DERs is through chance constraints since they can be used to bound the probability of voltage constraint violations. Chance constraints have been used extensively in power system operations, including~\cite{ZhangEtAl2011,WuEtAl2014, MartinezEtAl2014, WangYiShenEtAl2017}. A challenge in using these constraints is that they may result in difficult optimization problems, requiring algorithms to be designed on a case-by-case basis. The global optimality or even convergence of these algorithms are often hard to guarantee. A second challenge is that the actual probability distribution of the uncertainties is almost never known in practice and has to be estimated or approximated using historical samples, adding to the computational challenge~\cite{YuryEtAl2016}. \textcolor{black}{Recent work on data driven distributional robust optimization sheds light on convex reformulation of the chance constrained problem \cite{AhmedEtAl2018}. Solving the resulting deterministic convex program is computationally efficient but relies on the assumption that uncertainty has joint Gaussian distribution.}


In this paper, we first propose a chance constrained optimization framework for the voltage regulation problem. Then we show how the problem can be solved efficiently using historical samples without explicitly estimating or approximating the probability distribution. Lastly, we provide (minor) conditions on which the algorithm will be guaranteed to be optimal.

Formally, we think about the voltage regulation problem as a chance constrained cost minimization problem.\footnote{Our framework allows a variety of costs, for example, conservative voltage reduction~\cite{ZhuEtAl2016}.} {More specifically, we impose a single chance constraint on the whole system that captures system security. Even this constraint is intractable, in this work, we show how this single constraint can be used and approximated to capture uncertainties from all buses in the system using the linearized DistFlow model introduced in~\cite{Wu89}.}


We extend our previous work in \cite{PanLiAsilomar2017} and show that our proposed voltage control problem can be solved without the need to deploy mixed integer programming (MIP), which is usually used in scenario approximation of chance constraints~\cite{LuedtkeEtAl2008}. Even if the distribution of the uncertainty is unknown, we provide an algorithm that computes a descent direction using historical samples. Moving along this direction would lead to a local minimum. In addition, we show that if the true underlying uncertainty follows a wide family of probability distributions, the proposed algorithm will find the solution \textcolor{black}{for a given set of historical samples. This solution approaches the global optimum if the number of the samples is large.}
In summary, we make the following contributions to voltage control in distribution systems:
\begin{itemize}
	\item The uncertainties of the DER generations are correlated and expressed as a single chance constraint imposed onto the whole system.
	\item We propose an efficient and tractable descent algorithm by finding a valid descent direction at each iteration using historical samples, which avoids MIP formulation and cumbersome integral computation.
	\item We show that if the uncertainty in the DERs comes from a log-concave distribution, \textcolor{black}{the optimization problem is actually convex and the solution of our algorithm approaches global optimum almost surely as the number of samples goes to infinity.}
\end{itemize}

The rest of the paper is organized as follows. Section \ref{sec:distFlow} presents the modeling of the distribution network. Section \ref{sec:prob} proceeds with the formulation of the voltage control problem and demonstrate the robustness of the proposed framework with an illustrating example. In Section \ref{sec:sample_approximation}, we present the proposed descent algorithm to solve the optimization problem efficiently with samples. We further state that the problem is convex for a wide range of probabilistic distributions. Section \ref{sec:simu} validates the statement by simulation results. Finally Section \ref{sec:conc} concludes the paper.

\section{Preliminary: Distribution network} \label{sec:distFlow}
In this section we present the modeling of components in a radial distribution network in power systems. For interested readers, please refer to \cite{FarivarEtAl2012, GanetAl2012} for more details.
\subsection{Power flow model for radial networks}
We consider a distribution network with $N + 1$ buses collected in the set $ \{0, 1, . . . , N\}$. We also denote a line in the network by the pair $(i, k)$ of buses it connects. Bus $0$ is the feeder (reference bus).
For each line $(i, k)$ in the network, its impedance is denoted by $z_{ik} = r_{ik} + \text{j}x_{ik}$, where $r_{ik}$ and $x_{ik}$ is its resistance and reactance.
For each bus $i$, let $V_i$ be the voltage magnitude at bus $i$ and $s_i = p_i + \text{j}q_i$ be the complex power injection, i.e., the generation minus consumption. In addition, the subset $\mathcal{N}_{k}$ denotes bus $k$'s neighboring buses that are further down from the feeder head. The DistFlow equations \cite{Wu89} model the distribution network flow  for every line $(i,k) $ as:
\begin{align*}
P_{ik} - \sum_{l \in \mathcal{N}_k}P_{kl} & = - p_k + r_{ik} \frac{P_{ik}^2 + Q_{ik}^2}{V_i^2},\\
Q_{ik} - \sum_{l \in \mathcal{N}_k}Q_{kl} & = - q_k + x_{ik} \frac{P_{ik}^2 + Q_{ik}^2}{V_i^2},\\
V_i^2 - V_k^2  = &2(r_{ik}P_{ik} + x_{ik}Q_{ik})- (r_{ik}^2 + x_{ik}^2) \frac{P_{ik}^2 + Q_{ik}^2}{V_i^2},
\end{align*}
where $P_{ik},Q_{ik}$ are respectively the real and reactive power flow on line $(i,k)$. The term $\frac{P_{ik}^2 + Q_{ik}^2}{V_i^2}$ represents the loss in the power flow in line $(i,k)$.

\subsection{Linear approximation of the flow model}
A linear approximation of the DistFlow equations can be constructed~\cite{Wu89}. Assume that the losses is negligible and the voltage at each bus is close to 1. This enables us to approximate $V_i^2 - V_k^2 $ by $2(V_i - V_k)$ \cite{ZhuEtAl2016}. 
The linearized DistFlow model is given by :
\begin{subequations}
	\label{lindistflow}
	\begin{align}
	P_{ik} - \sum_{l \in \mathcal{N}_k}P_{kl} & = - p_k ,\\
	Q_{ik} - \sum_{l \in \mathcal{N}_k}Q_{kl} & = - q_k ,\\
	V_i - V_k & = r_{ik}P_{ik} + x_{ik}Q_{ik}.
	\end{align}
\end{subequations}
From \eqref{lindistflow}, we can write the voltage magnitude $\bm{V} = [V_1, \dots, V_N ]^{\top}$ in terms of reactive power injection $\bm{Q} = [Q_1, \dots, Q_N]^{\top}$ and real power injection $\bm{P} = [P_1, \dots, P_N]^{\top}$, and the reference voltage $V_0$ at the feeder:
\begin{equation}\label{system}
\bm{V} = \bm{R}\bm{P} + \bm{X}\bm{Q} + V_0
\end{equation}
where $\bm{R}, \bm{X} \in \mathbb{R}^{N \times N}$ are matrices with $R_{ik}$ and $X_{ik}$ as the element in $i^{th}$ row and $k^{th}$ column, respectively. The voltage profile at bus $1,..., N$ is denoted by $\bm{V} \in \mathbb{R}^N$.

Following the findings in \cite{LietAl2014}, we give the expressions of $R_{ik}$ and $X_{ik}$ in terms of line resistance $r_{ik}$ and reactance $x_{ik}$:
\begin{equation}
\label{RXelement}
R_{ik} = \sum_{(h,l) \in \mathcal{P}_i \cap \mathcal{P}_k} r_{hl}, \ X_{ik} = \sum_{(h,l) \in \mathcal{P}_i \cap \mathcal{P}_k} x_{hl},
\end{equation}
where $\mathcal{P}_i$ is the set of lines on the unique path from bus 0 to bus $i$.

\section{Voltage regulation with reactive power injection}\label{sec:prob}

To facilitate analysis, rewrite $\bm{V}$ as the difference between the bus voltage and reference voltage $V_0$, then \eqref{system} becomes:
\begin{equation}
\bm{V} = \bm{R}\bm{P} + \bm{X}\bm{Q}.
\end{equation}

\textcolor{black}{As renewables introduce uncertainty in the bus voltages, the voltage profile is reformulated into the following form:
	\begin{equation}\label{noisyV}
	\begin{aligned}
	\bm{V} & = \bm{R}(\bm{P}+\Delta \bm{P} ) + \bm{X}\bm{Q} \\
	& = \bm{R}\bm{P} + \bm{X}\bm{Q} + \bm{R}\Delta \bm{P} \\
	& = \bm{R}\bm{P} + \bm{X}\bm{Q} + \bm{\epsilon},
	\end{aligned}
	\end{equation}
	where $\Delta \bm{P}$ is the uncertain real power injection from the DER. We compactly write uncertainty imposed on voltage as $\bm{\epsilon}=\bm{R}{\Delta \bm{P}}$, which represents the uncertainty of the system. Note that even if $\Delta \bm{P}$ is independent across each bus, the resulting randomness in $\bm{\epsilon}$ can still be highly correlated.} 

Ideally, the voltages in the system should be maintained within a tight region (e.g., plus/minus 5\% of nominal). With uncertainties introduced by the operation of the DERs, we model this constraint in a probabilistic fashion. We use the following chance constraint which bounds the probability of the voltages staying in the prescribed bounds:
\begin{equation}\label{max00}
\Pr \{ \underline{\bm{V}}  \leq \bm{V} \leq \overline{\bm{V}} \} \geq \alpha,
\end{equation}
which is equivalent to be written as:
\begin{equation}\label{max0}
\Pr \{ \underline{\bm{V}}  \leq \bm{R}\bm{P} + \bm{X}\bm{Q} + \bm{\epsilon} \leq \overline{\bm{V}} \} \geq \alpha,
\end{equation}
where $\underline{\bm{V}} $ and $\overline{\bm{V}}$ are the voltage bounds. The value of $\alpha$ is a parameter that can be chosen to indicate the probability that event $\underline{\bm{V}}  \leq \bm{R}\bm{P} + \bm{X}\bm{Q} + \bm{\epsilon} \leq \overline{\bm{V}}$ occurs. %

\subsection{Main Optimization Problem}
In this paper we only consider reactive power regulation and assume that active load injection $\bm{P}$ is determined exogenously and the controllable variable is the reactive power injection $\bm{Q}$. Denote $\Pr \{ \underline{\bm{V}}  \leq \bm{R}\bm{P} + \bm{X}\bm{Q} + \bm{\epsilon} \leq \overline{\bm{V}} \}$ by $g(\bm{Q})$, for a given tolerance level $\alpha$, the centralized voltage regulation problem is then captured as the following:
\begin{subequations}\label{prob:main}
	\begin{align}
	&  \min_{\bm{Q}} \   C(\bm Q) \label{eqn:main_obj}\\
	\mbox{s.t.}\ & g(\bm{Q}) \geq \alpha, \label{eqn:alpha} \\
	& \underline{\bm{Q}} \leq \bm{Q} \leq \overline{\bm{Q}}, \label{eqn:Qlimit}
	\end{align}
\end{subequations}
where the cost function $C(\bm{Q})$ can be any convex cost function, for example, the 2 norm deviation of reactive power support $\Vert\bm{Q}\Vert_2$ as discussed in \cite{KekatosEtAl2015}. This cost function encourages small amount of reactive power support to maintain the acceptable voltage deviation due to uncertainty. 

The value of $\alpha$ indicates how the voltage profile behaves within the prescribed bounds. \emph{Risk level} is the usual adopted term with chance constrained optimization and it is equal to $1 - \alpha$, representing the severity of the system state.




\subsection{Existing Benchmark: Per-Bus Constraints}
Our approach is different from the existing literature when dealing with chance constraints. In most existing literature with randomness in the distribution network, chance constraints are introduced in~\cite{ZhangEtAl2011,WuEtAl2014, SaumardEtAl2014,YuryEtAl2016} separately for each dimension of the system as:
\begin{equation}
\begin{aligned}
&\Pr \{ \underline{V}_i  \leq V_i \leq \overline{V}_i \}\\
= & \Pr \{\underline{V}_i \leq \bm{R}_i^{\top}\bm{P} + \bm{X}_i^{\top}\bm{Q} + {\epsilon}_i \leq \overline{V}_i \} \geq  \eta_i,
\end{aligned}
\end{equation}
where $\bm{R}_i^{\top}$ and  $\bm{X}_i^{\top}$ extracts the $i$th row in respective matrices. The chance constraint at bus $i$ is associated with prescribed tolerance $\eta_i$. Assuming that each bus has the same tolerance, the optimization problem that incorporates per-bus chance constraint is in the following form:
\begin{subequations}\label{prob:per_bus}
	\begin{align}
	&  \min_{\bm{Q}} \   \textcolor{black}{\frac{1}{2}\Vert \bm{Q}\Vert_2^2}\\
	s.t.\ & g_i(\bm{Q}) \geq \eta, \forall i, \label{eqn:perbus}\\
	& \underline{\bm{Q}} \leq \bm{Q} \leq \overline{\bm{Q}},
	\end{align}
\end{subequations}
where $g_i(\bm{Q}) \overset{\Delta}{=}  \Pr \{\underline{V}_i \leq \bm{R}_i^{\top}\bm{P} + \bm{X}_i^{\top}\bm{Q} + {\epsilon}_i \leq \overline{V}_i \}$.

\textcolor{black}{Discussions on the difference between the per-bus constraints ( the feasible region depicted by \eqref{eqn:perbus}) and a single joint chance constraint (the feasible region depicted by~\eqref{eqn:alpha}) are provided in \cite{ZhangYuEtAl2013, RoaldEtAl2017, BakerEtAl2017}. The difference resides in finding the correct $\eta$ for each bus in \eqref{eqn:perbus} to approximate the original joint chance constraint and intractability of the constraint.} 
The small example below illustrates that the proposed single system chance constraint framework, captures the coupling between buses and is therefore more realistic and applicable. In addition, as we show in Section \ref{sec:sample_approximation}, the optimization problem can be solved efficiently, negating some of the computational difficulties with higher dimensional chance constraints. \\

\noindent \textbf{Toy Example} Consider a line network with 3 buses.
Suppose that the reactive power injections at bus 1 and 2 are limited to {0.1} p.u., we have a linear constraint as $-0.1 \leq Q_1, Q_2 \leq 0.1$. Then our proposed framework solves the following optimization problem:
\begin{equation}\label{eq:4bus1}
\begin{aligned}
&  \min_{\bm{Q}} \ \textcolor{black}{\frac{1}{2}\Vert \bm{Q}\Vert_2^2}\\
s.t. \
& g(\bm{Q}) \geq \alpha, \\
\   -0.1 &\leq Q_j \leq 0.1, \forall j \in \{1,2\}\\
\end{aligned}
\end{equation}
and its optimal solution is denoted by $\bm{Q}^*_a$. Here we take $\alpha$ to be $0.9$.

The per-bus formulation is the following optimization problem:
\begin{equation}\label{eq:4bus2}
\begin{aligned}
&  \min_{\bm{Q}} \ \textcolor{black}{\frac{1}{2}\Vert \bm{Q}\Vert_2^2} \\
s.t. \
& g_i(\bm{Q}) \geq  \eta, \forall i \in \{1,2\}\\
\   -0.1 &\leq Q_j \leq 0.1, \forall j \in \{1,2\}\\
\end{aligned}
\end{equation}
and its optimal solution of \eqref{eq:4bus2} by $\bm{Q}^*_p$. Suppose that the randomness $\bm{\epsilon}$ follows a Gaussian distribution with $\bm{\mu} =\begin{bmatrix}
0 \\ 0
\end{bmatrix}, \bm{\Sigma} = \begin{bmatrix}
0.002 & 0.0014 \\ 0.0014 & 0.006
\end{bmatrix}$, 
and $\bm{P}$ is randomly picked between $0.3$ and $-0.3$ p.u. Unlike the problem in \eqref{eq:4bus1}, setting a ``right'' $\eta$ in \eqref{eq:4bus2} is not straightforward. Suppose we want to achieve the same level confidence as \eqref{eq:4bus1} where the system operates within the prescribed bounds with probability at least $0.9$, then what is the right $\eta$ to take?
As suggested by previous studies~\cite{ZhangetAl2013}, a natural candidate for $\eta$ is to set it equal to $\sqrt[2]{0.9}=0.949$ by thinking of each bus as independent to each other. A second candidate is simply to set it at $0.9$, the same as $\alpha$.

The main results of the three-bus line network are shown in Table \ref{table:4bus}, where $\bm{Q}^*$ denotes the optimal solution for perspective frameworks. The value $g(\bm{Q}^*)$ denotes the probability $\Pr \{ \underline{\bm{V}}  \leq \bm{R}\bm{P} + \bm{X}\bm{Q}^* + \bm{\epsilon} \leq \overline{\bm{V}} \}$ and is the figure of merit we compare the solutions with. 
The bound on the voltage deviation is denoted by $\overline{\bm{V}} = [\ 0.05 \ 0.05]^{\top}$ and $\underline{\bm{V}} = [\ -0.05 \ -0.05]^{\top}$. As can be seen from Table \ref{table:4bus}, setting $\eta =\sqrt[2]{0.9}$ drives the optimization problem with the per-bus constraint infeasible. The second value of $ \eta = 0.9$ makes the problem feasible, but at the cost of lowering the joint probability
$g(\bm{Q}^*)$ to be 0.86. This result shows that it is difficult to set the correct tolerance level in the per-bus chance constraints. More details about the differences in the feasible regions are given in the expanded online version of the paper~\cite{LiEtAl2017}.
\begin{table}[!h]
	\renewcommand{\arraystretch}{1.3}
	\caption{The value of $g(\bm{Q}^*)$ under different frameworks.} 
	\centering
	\begin{tabular}{|c|c|c|c|}
		\hline
		\bfseries Framework & \eqref{eq:4bus1}, $\alpha = 0.9$ & \eqref{eq:4bus2}, $\eta = \sqrt[2]{0.9}$  & \eqref{eq:4bus2}, $\eta = 0.9$ \\
		\hline
		\bfseries $g(\bm{Q}^*)$ & 0.9 & Infeasible & 0.86  \\
		\hline
	\end{tabular}
	\label{table:4bus}
\end{table}

In the following section, we elaborate on the framework based on \eqref{prob:main} and show that it can be solved efficiently using samples.

\section{Sample-Based Descent Algorithm}\label{sec:sample_approximation}
The proposed optimization problem in \eqref{prob:main} has a convex objective, a box constraint and a chance constraint. However, this chance constraint is not easy to deal with. In addition, the probabilistic distribution of the uncertainties in the system is usually not directly known. Instead, \emph{historical observation samples} of the uncertainties (e.g., recorded solar generation values) are available. These samples can be used to find a probability distribution, but this approach is cumbersome to implement because of two reasons: 1) it is not obvious how to find a good parameterization of the probability distribution to learn from the existing samples; and 2) even if the distribution is given, solving problem \eqref{prob:main} with the chance constraint is non-trivial.


In this section, we present a sample-based algorithm that can efficiently solve the problem in \eqref{prob:main} directly from a set of historical samples, without the need to fit a distribution. Let $\mathcal{S}$ denote the set of samples for the random variable $\bm{\epsilon}$, which we sometimes write as $\bm{\epsilon}^{(s)}, s \in \mathcal{S} = \{1,2, \cdots, S\}$. In the rest of the section, we first provide the algorithm then show it is a descent algorithm. Throughout this section, we use $(\hat{\cdot})$ to represent the approximation of a quantity.


\subsection{Find a descent direction}\label{subsec_random_search}

The first step in solving the optimization problem in \eqref{prob:main} is to relax the constraints using the log barrier method~\cite{Boyd2004}:
\begin{equation}\label{eqn_relaxed}
\begin{aligned}
&  \min_{\bm{Q}} \   \mathcal{L}(\bm{Q}) = \textcolor{black}{\frac{1}{2}\Vert \bm{Q}\Vert_2^2} - \frac{1}{t}\log( - \alpha +  g(\bm{Q})) -\\
& \sum_i\frac{1}{t}\left(\log(\overline{\bm{Q}}_i - \bm{Q}_i) + \log(\bm{Q}_i - \underline{\bm{Q}}_i) \right) ,
\end{aligned}
\end{equation}
where $t$ is a tunable parameter, $(\cdot)_i$ represents the $i$th element of a vector, and $g(\bm{Q})$ is the chance constraint $\Pr \{ \underline{\bm{V}}  \leq \bm{R}\bm{P} + \bm{X}\bm{Q} + \bm{\epsilon} \leq \overline{\bm{V}} \}$. This optimization problem guarantees that the solution is always feasible to \eqref{prob:main} and the problem is equivalent to \eqref{prob:main} if we take $t \rightarrow \infty$.

To solve \eqref{eqn_relaxed}, the natural method to use is gradient descent. The gradient of $\mathcal{L}(\bm{Q})$ is:
\begin{equation} \label{eqn:grad}
\begin{aligned}
\nabla \mathcal{L}(\bm{Q}) &=   \textcolor{black}{\bm{Q}} - \frac{1}{t}\nabla (\log (g(\bm{Q}) - \alpha)) \\
& +  \frac{1}{t}(\overline{\bm{Q}} - \bm{Q})^{-1} - \frac{1}{t}( \bm{Q} -\underline{\bm{Q}})^{-1}.
\end{aligned}
\end{equation}

It turns out that directly using \eqref{eqn:grad} to solve \eqref{eqn_relaxed} is difficult since there is no easy way to find the gradient of the $(\log (g(\bm{Q}) - \alpha))$. In fact, even if the true distribution of $\bm{\epsilon}$ is known, finding $\nabla(\log (g(\bm{Q}) - \alpha))$ requires evaluating multidimensional integrals, which is intractable for most distributions. The key step to mitigate this difficulty is to find some descent direction which approximates the gradient using historical samples, rather than trying to compute the exact gradient. In this section we first describe how to find a proxy for $\nabla \log (g(\bm{Q}) - \alpha)$. Then in Section \ref{subsec_sample_approximation}, we show how it can be calculated from the samples.

The idea to find a descent direction is simple. Given any differentiable function $f:\mathbb{R}^N \rightarrow \mathbb{R}$, let $\bm{e}$ be a random unit vector in $\mathbb{R}^N$ and let $\Delta$ be a positive number. Then for any $\bm x$ in the domain of $f$, we can compute
\begin{equation}\label{eqn:p}
\bm p = \frac{f(\bm{x} + \Delta \cdot \bm{e}) - f(\bm{x})}{\Delta} \bm{e}.
\end{equation}

We can think of $\bm p$ as a random version of the gradient and will be taken as the direction of update in an optimization problem. To be a valid direction, we need to show that $\bm p$ is aligned (making an angle of less than $90^\circ$) with the actual gradient of $f$. This is given by the next theorem:
\begin{theorem}\label{thm_random_search}
	Let $f(\bm{x}):\mathbb{R}^N \rightarrow \mathbb{R}$ be a differentiable function, $\bm{e}$ be a random unit vector in $\mathbb{R}^N$ and $\Delta$ be a positive scalar. Define $\bm{p} = \frac{f(\bm{x} + \Delta \cdot \bm{e}) - f(\bm{x})}{\Delta} \bm{e}$, then $\nabla f(\bm{x})^{\top} \bm{p} > 0$ for a small enough $\Delta$ with probability 1 {when $\nabla f(\bm{x}) \neq 0$}.
\end{theorem}

The proof of Theorem~\ref{thm_random_search} is given in the Appendix. Since $\bm{p}$ is no more than $90^\circ$ degrees apart from the gradient $\nabla f(\bm{x})$ (the dot product $\nabla f(\bm{x})^{\top} \bm{p}$ is positive), the next theorem shows that descending according to $\bm{p}$ is sufficient to reach a local minimum:
\begin{theorem}\label{thm:descent}
	Let $f(\bm{x}):\mathbb{R}^N \rightarrow \mathbb{R}$ be a differentiable function and $\bm{p}_{\bm x}$ be a vector such that {$\nabla f(\bm{x})^{\top} \bm{p}_{\bm x} > 0$}. Suppose $\bm x$ has the following update rule:
	\begin{equation} \label{eqn:update}
	\bm{x}_{m+1}= \bm{x}_m - \lambda_m \bm{p}_{\bm x_m},
	\end{equation}
	then the sequence of $\bm{x}_m$ converges to \textcolor{black}{a stationary point of $f$ for appropriately chosen step sizes $\lambda_m$}.
\end{theorem}

Theorem \ref{thm:descent} states that as long as we move in some descent direction, then we \textcolor{black}{will find a stationary point. Note that the randomness in the samples used to compute the direction would allow it to escape saddle points (similar to stochastic gradient algorithm~\cite{ge2015escaping}). Later in the section, we show that our problem of interest is convex, therefore all stationary points are global minimum.} The step sizes $\lambda_m$ can be found by many methods, including using fixed step sizes if the function is Lipschitz, using a decreasing sequence, using backtracking and others~\cite{bottou2010large}. The proof of Theorem~\ref{thm:descent} proceeds almost identically as proving that the gradient descent converges to a stationary point, and can be found for example in~\cite{ErmolievEtAl1994}.

Together, Theorems~\ref{thm_random_search} and \ref{thm:descent} tell us that as long as we can compute the function values of $f$, then we can use a descent algorithm to find a local optima. The next section describes how we can find the function values of the unconstrained problem in \eqref{eqn_relaxed} efficiently using samples.

\subsection{Sample Approximation}\label{subsec_sample_approximation}

Note that $\mathcal{L}(\bm{Q})$, the objective in the unconstrained problem in \eqref{eqn_relaxed}, is a differentiable function since $\bm{\epsilon}$ has a continuous density function. Then the update rule discussed in Theorem \ref{thm:descent} for $\mathcal{L}(\bm{Q})$ at $m$th iteration is:
$\bm{Q}_{m+1} = \bm{Q}_m - \lambda_m \bm{p}_{\bm Q_m},$
where $\bm{p}_{\bm Q_m}$ should be a proxy for $\nabla \mathcal{L}(\bm{Q})$ in \eqref{eqn:grad} that satisfies Theorem \ref{thm_random_search}. In this section, we focus on how to construct $\bm{p}_{\bm Q_m}$ using samples $\bm{\epsilon}^{(s)}$.

To start with, the difficult term to compute in \eqref{eqn:grad} is $\log (g(\bm{Q})-\alpha)$, and we focus on this term here. Recall that
$g(\bm Q)= \Pr\{\underline{\bm{V}} \leq \bm{RP} +\bm{XQ} + \bm{\epsilon} \leq \overline{\bm{V}}\}$ and we write the right hand side as $\Pr(\bm{RP +XQ})$ for brevity in this section. It is easy to compute a sample approximation of $\Pr(\bm{RP +XQ})$ at a particular $\bm{Q}$.
Denote the sample approximation of function $g(\bm{Q})$ by $\hat{g}(\bm{Q}, \{\bm{\epsilon}^{(s)}, \forall s \in \mathcal{S}\})$, and let $\bm{x} = \bm{RP} + \bm{XQ}$, then define
\begin{equation}\label{eqn_sample_approx_g}
\begin{aligned}
& \hat{g}(\bm{Q}, \{\bm{\epsilon}^{(s)}, \forall s \in \mathcal{S}\}) \\
\overset{\Delta}{=}  &   \frac{\sum_{s = 1}^{S} \mathbbm{1}\{ \underline{\bm{V}} \leq \bm{R}\bm{P} + \bm{X}\bm{Q} + \bm{\epsilon}^{(s)} \leq \overline{\bm{V}}\}}{S}\\
= &  \frac{\sum_{s = 1}^{S} \mathbbm{1}\{ \underline{\bm{V}} \leq \bm{x} + \bm{\epsilon}^{(s)} \leq \overline{\bm{V}}\}}{S}\\
\overset{\Delta}{=} &  \widehat{\Pr}(\bm{x}, \{\bm{\epsilon}^{(s)}, \forall s \in \mathcal{S}\}),
\end{aligned}
\end{equation}
where $S$ is the size of the set $\mathcal{S}$ and $\widehat{\Pr}(\bm{x},\{\bm{\epsilon}^{(s)}, \forall s \in \mathcal{S}\})$ is the sample approximation of $\Pr(\bm{x})$. The quantity in \eqref{eqn_sample_approx_g} can be evaluated quickly since it only involves $SN$ comparisons using samples $\bm{\epsilon}^{(s)}$ and a particular value of $\bm{Q}$. As the sample size grows, $\widehat{\Pr}(\bm{x},\{\bm{\epsilon}^{(s)}, \forall s \in \mathcal{S}\})$ becomes better at approximating $\Pr(\bm{x})$~\cite{HsuEtAl1947}.

Now we can find an approximation of a descent direction for $\log ({g}(\bm{Q})-\alpha)$. First, using the chain rule, the gradient of $\log ({g}(\bm{Q})-\alpha)$ is given by
\begin{equation}\label{eqn_chain_rule}
\nabla \log  (g(\bm{Q})-\alpha)
=  \bm{X}^{\top}\frac{ \nabla \Pr(\bm{x})}{g(\bm{Q})-\alpha},
\end{equation}
where ${\bm{x} = \bm{RP} + \bm{XQ}}$.

Then we introduce the following corollary of Theorem~\ref{thm_random_search}:
\begin{corollary}\label{cor:compos}
	Assume that $f({x}): \mathbb{R} \rightarrow \mathbb{R}$ and $g(\bm{y}): \mathbb{R}^N \rightarrow \mathbb{R}$. If $f(\cdot)$ and $g(\cdot)$ are both differentiable, then {$\nabla_{\bm{y}} f(g(\bm{y}))^{\top} \bm{p} > 0$ with probability 1 when $\nabla_{\bm{y}} f(g(\bm{y})) \neq 0$}, where $\bm{p} = \frac{\partial f({x})}{\partial {x}}\frac{g(\bm{y} + \Delta \cdot \bm{e}) - g(\bm{y})}{\Delta} \bm{e}|_{x = g(\bm{y})}$, for a sufficiently small positive scalar $\Delta$  and $\bm{e}$ is a random unit vector in $\mathbb{R}^N$.
\end{corollary}

Combining the chain rule described in \eqref{eqn_chain_rule} with Corollary~\ref{cor:compos} and sample approximation as defined in~\eqref{eqn_sample_approx_g}, we can find a (random) descent direction of $\log (g(\bm{Q})-\alpha)$ as:
\begin{equation} \label{eqn:approx_log}
\begin{aligned}
& \bm{X}^{\top} \bm{e} \frac{ \widehat{\Pr}(\bm{x} + \Delta \cdot \bm{e}, \{\bm{\epsilon}^{(s)}, \forall s \in \mathcal{S}\}) - \widehat{\Pr}(\bm{x},\{\bm{\epsilon}^{(s)}, \forall s \in \mathcal{S}\})}{\Delta} \\
\cdot & \frac{1}{\hat{g}(\bm{Q}, \{\bm{\epsilon}^{(s)}, \forall s \in \mathcal{S}\})-\alpha}\\
\overset{\Delta}{=} &  
\frac{\widehat{\nabla} g(\bm{Q},  \{\bm{\epsilon}^{(s)}, \forall s \in \mathcal{S}\})}{\hat{g}(\bm{Q}, \{\bm{\epsilon}^{(s)}, \forall s \in \mathcal{S}\})-\alpha},
\end{aligned}
\end{equation}
where $\bm{e}$ is a random unit vector in $\mathbb{R}^{n}$, $\Delta$ is a small positive number and ${\bm{x} = \bm{RP} + \bm{XQ}}$.

This approximation $\widehat{\nabla} g(\bm{Q},  \{\bm{\epsilon}^{(s)}, \forall s \in \mathcal{S}\})$ is {inspired by \emph{stochastic quasi-gradient} (SQG) \cite{ErmolievEtAl1988} of $g(\bm{Q})$}. It was initially brought up to solve stochastic programs where the objective function is hard to evaluate. In our problem, we adopt the similar idea and approximate the gradient which is hard to compute. {Note that if the distribution of $\bm{\epsilon}$ is known and can be efficiently sampled from each time we compute the gradient, the approximation is SQG itself.} 


With the introduction of {such approximation} on $g(\bm{Q})$ and the descent direction of $ \log  (g(\bm{Q})-\alpha)$ in \eqref{eqn:approx_log}, a (random) descent direction of ${\mathcal{L}}(\bm{Q})$ based on samples $\bm{\epsilon}^{(s)}$ is given by:
\begin{equation*}
\begin{aligned}
\widehat{\nabla} \mathcal{L}(\bm{Q},  \{\bm{\epsilon}^{(s)}, \forall s \in \mathcal{S}\}) \overset{\Delta}{=}
& \bm{Q} - \frac{1}{t} \frac{\widehat{\nabla} g(\bm{Q},  \{\bm{\epsilon}^{(s)}, \forall s \in \mathcal{S}\})}{\hat{g}(\bm{Q}, \{\bm{\epsilon}^{(s)}, \forall s \in \mathcal{S}\})-\alpha} \\
&+  \frac{1}{t}(\overline{\bm{Q}} - \bm{Q})^{-1} - \frac{1}{t}( \bm{Q} -\underline{\bm{Q}})^{-1},
\end{aligned}
\end{equation*}
which is a valid $\bm{p}_{\bm{Q}}$ for the update rule.

\subsection{Main Algorithm}
Now we plug in $\widehat{\nabla} \mathcal{L}(\bm{Q},  \{\bm{\epsilon}^{(s)}, \forall s \in \mathcal{S}\})$ into the standard log barrier method and arrive at the proposed descent algorithm {using sample approximation}, shown in Algorithm \ref{algo:log_barrier_quasi}. In Algorithm \ref{algo:log_barrier_quasi}, we have an outer iteration that increases the value of $t$, which forces the optimization problem in \eqref{eqn_relaxed} to be closer to the true optimization problem in \eqref{prob:main}. In the inner iteration when $t$ is fixed, we run the loop till convergence.

\begin{algorithm}[h]
	\caption{Modified log barrier method {using samples}.}
	\label{algo:log_barrier_quasi}
	\begin{algorithmic}[1]
		\State \textbf{Input}: $t_0>0$, $\lambda_0>0$, $\varepsilon_1,\varepsilon_2,\varepsilon_3 >0$, $\tau>1$, $\phi<1$, a feasible $\bm{Q}_0$, $m = k = 1$.
		\While{$\frac{1}{t_k} \geq \varepsilon_1$ and $\lambda_{k} \geq \varepsilon_2$}
		\State $t_k$ = $\tau t_{k-1}$, $\lambda_{k} = \phi \lambda_{k-1}$.
		\While{$\frac{\Vert \bm{Q}_{m} - \bm{Q}_{m - 1}\Vert_2^2}{2} \geq \varepsilon_3$}
		\parState{Compute gradient w.r.t. random sample $\{\epsilon^{(s)}\}$ at $t = t_k$, denote it by $\widehat{\nabla} \mathcal{L}(\bm{Q}_{m-1},  \{\bm{\epsilon}^{(s)}, \forall s \in \mathcal{S}\})$.}
		\parState{Do backtracking until obtaining feasible $\bm{Q}_m = \bm{Q}_{m - 1} - \lambda'\widehat{\nabla} \mathcal{L}(\bm{Q}_{m-1},  \{\bm{\epsilon}^{(s)}, \forall s \in \mathcal{S}\})$, where $\lambda'$ is determined by backtracking. Let $\lambda_k = \lambda'$.}
		\State  $m = m + 1$.
		\EndWhile
		\State $k = k + 1$.
		\EndWhile\label{quasiwhile}
		\State \textbf{Output} $\bm{Q}_{m}$.
	\end{algorithmic}
\end{algorithm}

In addition, since $\widehat{\nabla} \mathcal{L}(\bm{Q}_{m-1},  \{\bm{\epsilon}^{(s)}, \forall s \in \mathcal{S}\})$ is random and the chances that it is exactly orthogonal to the true gradient is zero, Algorithm \ref{algo:log_barrier_quasi} is guaranteed to converge, as given in Theorem~\ref{thm:descent}. In the simulations, we show that Algorithm \ref{algo:log_barrier_quasi} is more efficient than the conventional MIP on real datasets. \textcolor{black}{Note that in Algorithm \ref{algo:log_barrier_quasi} we use log-barrier method as an example to solve the optimization problem in \eqref{prob:main}. The essence is to use a few samples to empirically evaluate the gradient and deploy gradient descent algorithms. As a result, one benefit of such algorithm is that it can incorporate additional constraints, and can be extended to tools like ADMM for parallel computation as well.}


\subsection{Convexity of the optimization problem}
\textcolor{black}{In this section, we show that the chance constraint $g(\bm{Q}) \geq \alpha $ is convex for a broad family of probabilistic distributions. Since all other constraints in \eqref{prob:main} are also convex, the optimization problem is convex and Algorithm \ref{algo:log_barrier_quasi} converges to the \textcolor{black}{global optimum almost surely} when we have infinitely many samples. This is formally shown in Theorem \ref{theorem0}}.

\begin{theorem}\label{theorem0}
	\textcolor{black}{We assume that the uncertainty $\bm{\epsilon}$ has a continuous log-concave probabilistic distribution. Let $\bm{\epsilon}^{(s)}$'s represent realizations of uncertainty $\bm{\epsilon}$ and let $\{\{\bm{\epsilon}^{(s)}, s \in \mathcal{S}_1\}, \{\bm{\epsilon}^{(s)}, s \in \mathcal{S}_2\}, \cdots\}$ be a set of sets of realizations, where the size of $\mathcal{S}_j$ is going to infinity~\footnote{Note that we do not require the sets to be disjoint. For example, $\{\{\bm{\epsilon}^{(s)}, s \in \mathcal{S}_1\}, \{\bm{\epsilon}^{(s)}, s \in \mathcal{S}_2\}, \cdots\}$ can be formed by adding new realizations to $\{\bm{\epsilon}^{(s)}, s \in \mathcal{S}_1\}$.}. We further let $\mathcal{P}^{(j)}$ be the empirical distribution defined by the points in $\{\bm{\epsilon}^{(s)}, s \in \mathcal{S}_j\}$, and $\bm{Q}^{(j)}$ be obtained according Algorithm~\ref{algo:log_barrier_quasi}. Then using the results in \cite{DupacovaEtAl1988,Shapiro1993,KingEtAl1993,Robinson1996}, if $\mathcal{P}^{(j)}$ approaches $\mathcal{P}$ induced by $\bm{\epsilon}$ in distribution, then $\bm{Q}^{(j)}$ converges to $\bm{Q}^*$ and $L(\bm{Q}^{(j)})$ converges to $L(\bm{Q}^*)$ almost surely, where $L(\cdot)$ is defined in \eqref{eqn_relaxed} and $\bm{Q}^*$ is the optimal solution to \eqref{eqn_relaxed}. When $t \rightarrow \infty$ in \eqref{eqn_relaxed}, $\bm{Q}^*$ is also the optimal solution of \eqref{prob:main}.}
\end{theorem}

{Log-concave distributions include many of the commonly encountered distribution in practice, including the joint Gaussian, gamma, uniform, logistic, Laplace and others~\cite{An1996}. For example, the forecast errors of solar and wind are usually assumed to follow one of these distributions. Chance constraints involving log-concave random variables have been discussed extensively in literature, i.e., see \cite{WetsEtAl1983, Prekopa1995, VanEtAl2010, XieEtAl2017} and the references within. In this paper, we specifically focus on the form of rectangular chance constraints involving log-concave distribution. We refer to \cite{WetsEtAl1983, Prekopa1995} on general discussions and provide a concise and simple proof of Theorem \ref{theorem0} in the appendix.} 

		\section{Simulation}\label{sec:simu}
	In this section, we validate Algorithm~\ref{algo:log_barrier_quasi} with the IEEE standard 123 bus distribution system. The solver for {branch and bound (BB) algorithm} in MIP is GUROBI \cite{GUROBI}. {We adopt default Gurobi parameters with feasibility tolerance as 1e-9 and set time limit to be 1800 seconds. Since optimality gap is set to be 1e-10 in Gurobi as default, we set $\varepsilon_1, \varepsilon_2, \varepsilon_3$ to be 1e-10 to compare our algorithm against BB in Gurobi.} We use Matlab on MacBook Pro with 2.7 GHz Intel Core i5 to conduct all the simulations. First, we introduce the benchmark state-of-the-art method we compare Algorithm~\ref{algo:log_barrier_quasi} against. Then we give more results comparing Algorithm 1 and BB on the 123 bus system.

	\subsection{Sample approximation benchmark}\label{subsec:SAA}
	A natural way to approximate the chance constraint is to use indicator variables to replace the probability constraint with samples~\cite{LuedtkeEtAl2008}:
	\begin{subequations}\label{prob:main_SAA}
		\begin{align}
		&  \min_{\bm{Q}} \  \textcolor{black}{\frac{1}{2}\Vert \bm{Q}\Vert_2^2} \label{eqn:main_obj_SAA}\\
		s.t.\ & \sum_{s \in \mathcal{S}} \mathbbm{1} \{\underline{\bm{V}} \leq \bm{R}\bm{P} + \bm{X}\bm{Q} + \bm{\epsilon}^{(s)} \leq \overline{\bm{V}}\} \geq \alpha S, \label{eqn:alpha_SAA} \\
		& \underline{\bm{Q}} \leq \bm{Q} \leq \overline{\bm{Q}},
		\end{align}
	\end{subequations}
	where $ \mathbbm{1}\{\cdot\}$ is the indicator function.
	
	In \eqref{prob:main_SAA}, we use the empirical estimation $\sum_{s \in \mathcal{S}} \mathbbm{1} \{\underline{\bm{V}} \leq \bm{R}\bm{P} + \bm{X}\bm{Q} + \bm{\epsilon}^{(s)} \leq \overline{\bm{V}}\}/S $ to replace the true probability $\Pr \{ \underline{\bm{V}} \leq \bm{R}\bm{P} + \bm{X}\bm{Q} + \bm{\epsilon} \leq \overline{\bm{V}}\}$. However, the indicator function $\mathbbm{1}\{\cdot\}$ is discontinuous and non differentiable. To solve the problem in \eqref{prob:main_SAA}, introduction of binary variables for each scenario $s$ is necessary. Each binary variable indicates whether the voltage bound is violated or not given the particular sample $s$. We can then reformulate the problem as a MIP problem. In this paper, for comparison we actually implement an improved version of \eqref{prob:main_SAA} where the number of binary constraints is reduced by $\alpha$.
	Details are given in~\cite{LiEtAl2017}.

	\subsection{IEEE 123 bus system with renewable integration}
	In this section, we validate the statements by IEEE standard test feeder. Here we use IEEE 123 bus feeder \cite{IEEE123busref} as an example. The test feeder is shown in Fig. \ref{123bus}. We assume that bus 149 is the reference bus.
	\begin{figure}[!ht]
		\centering
		\includegraphics[width=1.05\columnwidth]{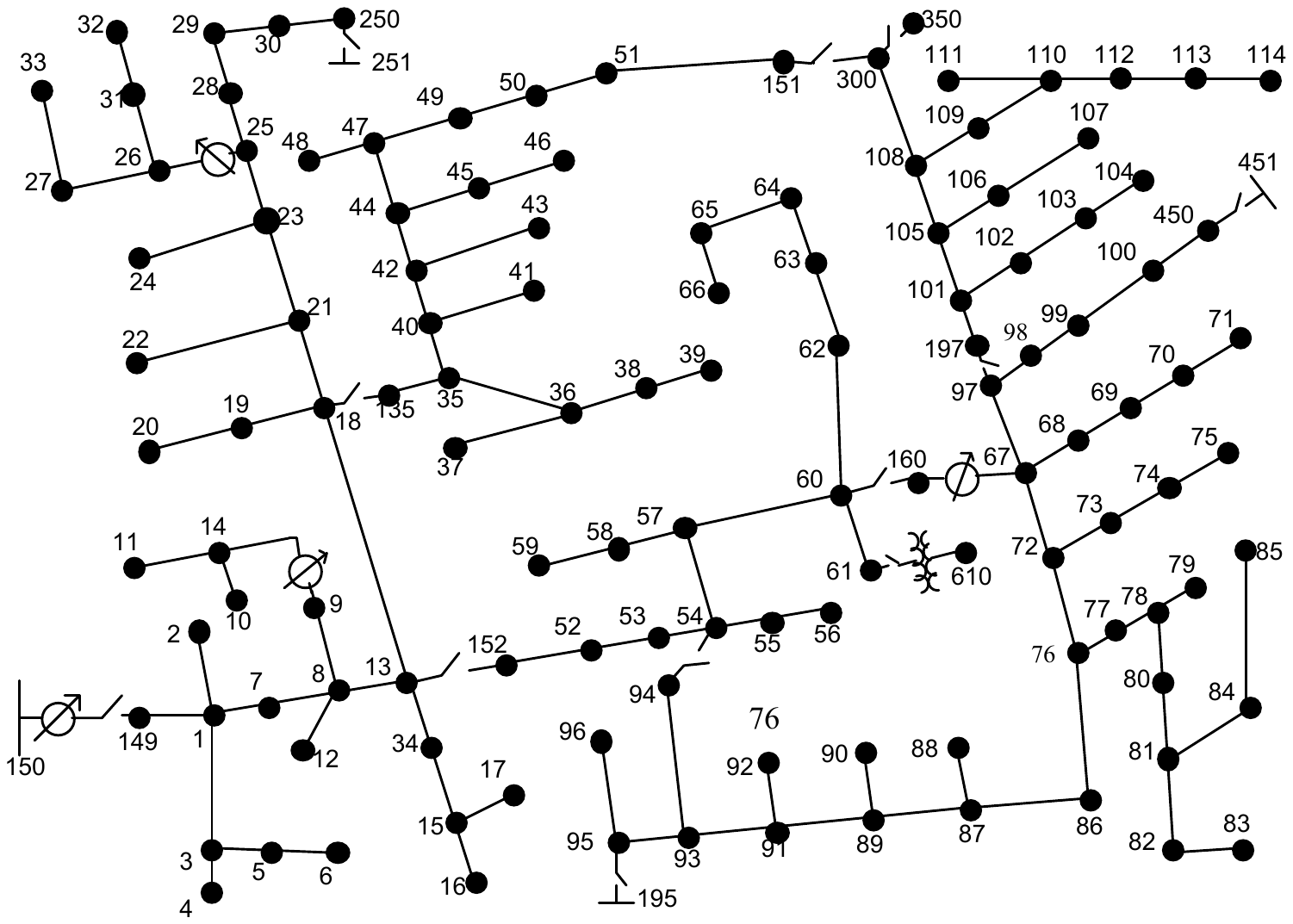}
		\caption{Schematic diagram of IEEE 123 bus test feeder.}
		\label{123bus}
	\end{figure}
	
	\textcolor{black}{The randomness of the system, $\Delta \bm{{P}}$, comes from the fluctuation of the renewable generation.} We use solar generation data from NREL \cite{nrel}. Its characteristic is shown in Fig. \ref{renewable_rand}. Besides solar energy fluctuations, the system is configured such that the switches between buses $13$ and $152$, $18$ and $135$, $54$ and $94$, $97$ and $197$ are closed.
	In addition, we take real power injection as a random vector. We assume that the next scheduling period is 2 p.m. The renewable generation is collected between 2 p.m. - 3 p.m. with 5 minute interval over a past whole year. We set the empirical risk level as $9\%$ (to be conservative), which means that $\alpha = 0.91$ on the samples. The deviation on voltage is set to be no more than $5\%$ of the nominal voltage and the bounds on $\bm{Q}$ is set to be 0.01 p.u.
	
	\begin{figure}[ht]
		\centering
		\includegraphics[width=0.8\columnwidth]{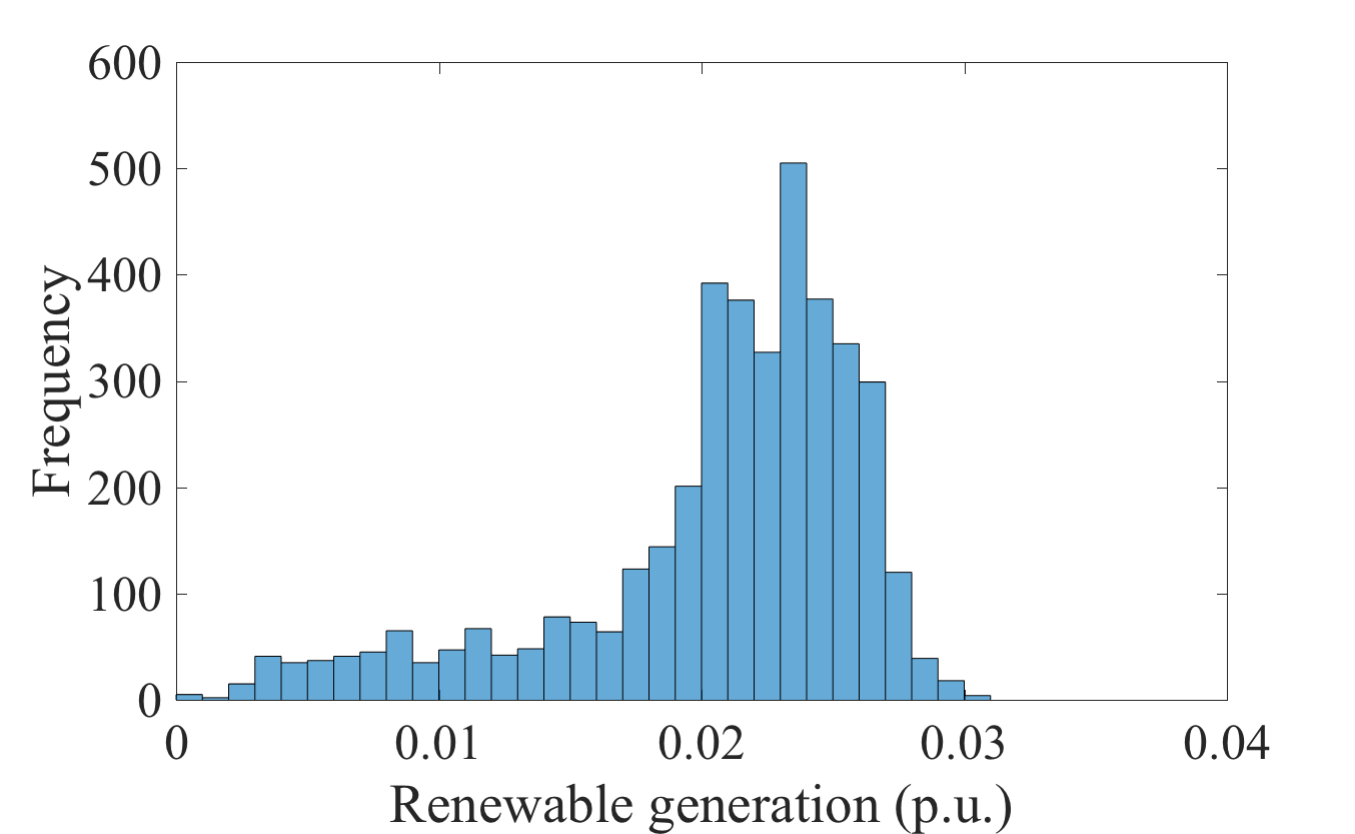}
		\caption{Histogram of PV generation between 2 p.m. and 3 p.m..}
		\label{renewable_rand}
	\end{figure}

	{The results on the 123 bus system are shown in Fig. \ref{122bus_decrease} and Table \ref{table:123bus}. In Fig. \ref{122bus_decrease}, we validate our proposed descent algorithm by plotting the norm of $\bm{Q}$ at each iteration. As can be clearly seen, $\Vert \bm Q \Vert_2$ is effectively decreasing with the proposed algorithm and stabilizes, which means that the algorithm converges to a local optimum.}
	
	\begin{figure}[ht]
		\centering
		\includegraphics[width=0.8\columnwidth]{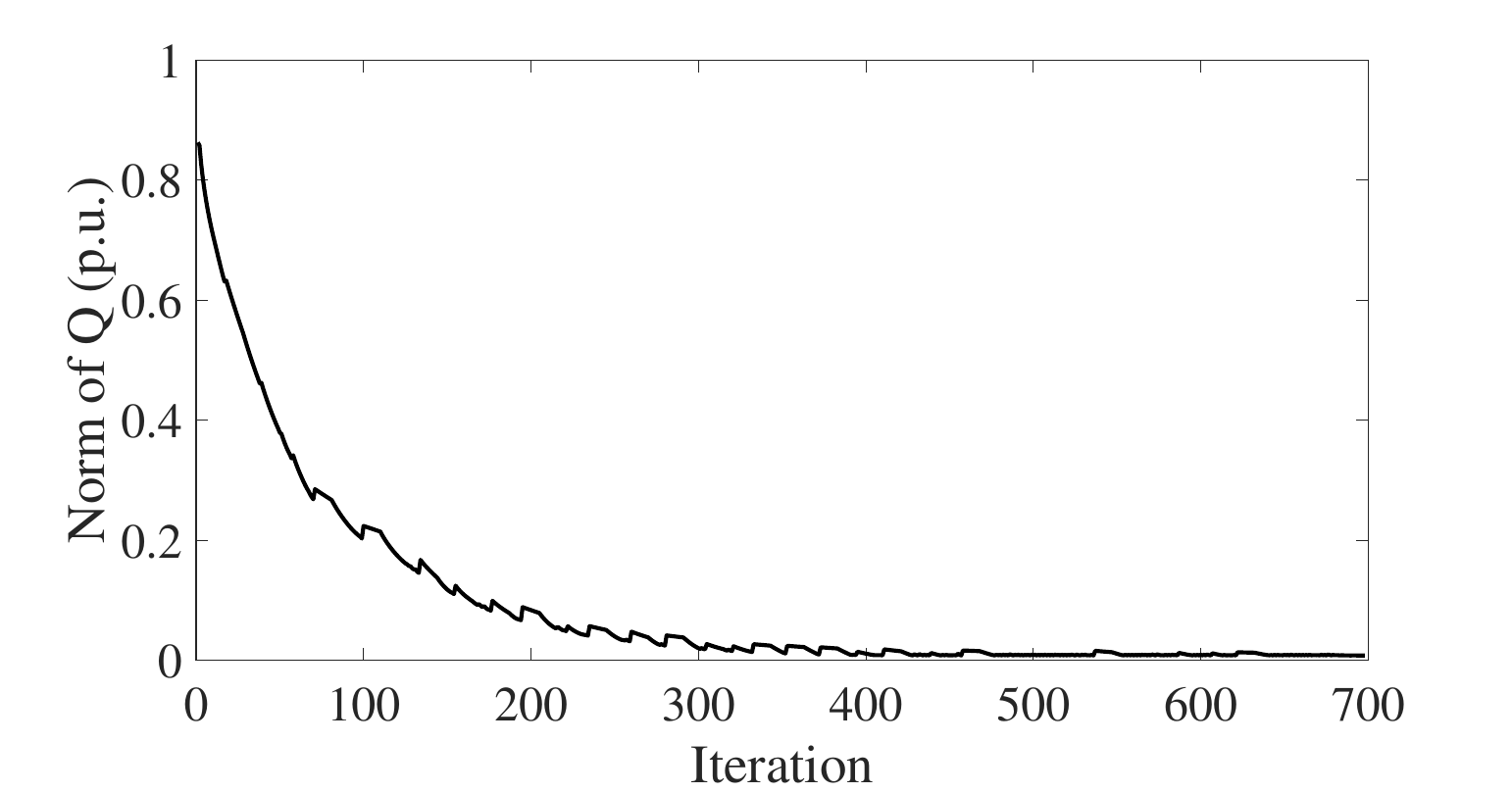}
		\caption{Decrease in $\Vert \bm{Q} \Vert_2$ using the proposed algorithm in IEEE 123 bus system.}
		\label{122bus_decrease}
	\end{figure}
	
	From Table \ref{table:123bus}, we see that BB takes a significantly longer time to solve as compared to Algorithm \ref{algo:log_barrier_quasi}, which cannot be adopted in real time dispatch decisions (i.e., 5 minutes). In addition, it takes less than {27 seconds in our algorithm to find a feasible $\bm{Q}$ with $\Vert \bm Q \Vert_2 \leq 0.009$, where it takes more than 400 seconds for BB to locate such $\bm{Q}$. Even with the help of heuristics by Gurobi, it still takes 164 seconds for Gurobi to find $\bm{Q}$ with $\Vert \bm Q \Vert_2 \leq 0.009$.} Both methods are able to achieve the desired risk level, i.e., around $9\%$ after sufficient iterations. The optimal $\bm{Q}$ returned by Algorithm \ref{algo:log_barrier_quasi} only differs in less than $5\%$ in norm as compared to the solution returned by BB. {To test against overfitting, we further split the samples into half and cross validated the results. We found the risk level and norm of $\bm{Q}$ is subject to a small difference (less than 5\%) using different samples, therefore the algorithm does not suffer from overfitting since the samples are sufficient enough to represent well the randomness in the system.}
	
	\begin{table}[!ht]
		\renewcommand{\arraystretch}{1.3}
		\caption{Comparison between BB and Algorithm \ref{algo:log_barrier_quasi} for IEEE 123 bus with renewable generation.}
		\label{table:123bus}
		\centering
		\begin{tabular}{|p{4cm}|c|c|}
			\hline
			\bfseries   &\bfseries BB  & \bfseries Algo. \ref{algo:log_barrier_quasi}\\
			\hline
			Running time (seconds)  & {$710$}  & $38$\\
			\hline
			Empirical risk level & $9.00\%$  & $8.97\%$\\
			\hline
			$\Vert\bm{Q}\Vert_2$ (p.u.)  & 0.0083 &  0.0086 \\
			\hline
			Most sensitive bus   & 104 &  104 \\
			\hline
			Number of buses with
			non-trivial reactive power support   & 41 &  41 \\
			
			\hline
		\end{tabular}
	\end{table}
	
	In addition, from the solution of BB, we find that the significant amount of reactive power support occurs at buses that are further down the feeder, i.e., bus 70, 71 and bus 104. These buses are more vulnerable towards uncertainty in the system thus more reactive power support is needed to avoid large voltage deviation. We find that Algorithm \ref{algo:log_barrier_quasi} locates those buses in a much faster computational time. Both algorithms are able to find 41 such \textcolor{black}{buses. 
		The} most reactive power support occurs at bus 104, with $Q_{104} = -0.0015$ p.u. as in BB and $Q_{104} = -0.0011$ p.u. as in Algorithm \ref{algo:log_barrier_quasi}. These reactive power support serve to mitigate the impact (over-voltage) by real power injection of random renewables.

		\section{Conclusion}\label{sec:conc}

		In this paper we adopt a stochastic framework to formulate voltage control problems with uncertainty. Compared to existing literature, we use a single chance constraint to capture the uncertainty in the distribution system. We show that this formulation is more realistic and less conservative than placing constraints on every bus in the distribution network. We also propose a tractable algorithm to solve the optimization problem without explicitly knowing the underlying probabilistic distribution. \textcolor{black}{The proposed algorithm solves the problem efficiently given a set of samples. For a wide range of probabilistic distributions, this solution converges to global optimum almost surely when number of samples goes to infinity because the problem is convex.} Simulation results validate our statements by standard IEEE test feeders.

	\bibliographystyle{IEEEtran}	
	\bibliography{ref}

	\appendix
\subsection{Discussion on convexity of the optimization problem}\label{sec:optprob}
{In Theorem \ref{theorem0}, we assume that the randomness comes from a specific family of distribution.} We now formally define log-concavity.
\begin{definition}
	A non-negative function $f:\mathbb{R}^N \rightarrow \mathbb{R}_{+}$ is logarithmically concave (or log-concave for short) if its domain is a convex set, and if $\log(f(x))$ is concave. 
\end{definition}

To prove Theorem \ref{theorem0}, we introduce two lemmas that discuss log-concavity of certain functions. {These two lemmas are derived easily following Chapter 4 in \cite{Prekopa1995}.} Lemma \ref{theorem2} states that the accumulated mass of a log-concave probabilistic function over a {rectangular box} is log-concave. Lemma \ref{theorem3} states that applying linear transformation to the variable in a log-concave function yields another log-concave function, given that the linear transformation has full row rank.

\begin{lemma}\label{theorem2}
	Denote $F(\bm{z}) = \int_{\bm{z}-\bm{u}}^{\bm{z}} f(\epsilon_1,\dots,\epsilon_N) \,d\epsilon_1 \dots d\epsilon_N$ {where $\bm{u} > 0$ is a fixed vector}. If the distribution $f(\epsilon_1,\dots,\epsilon_N)$ is log-concave in $\bm{\epsilon}$, then $F(\bm{z})$ is log-concave in $\bm{z}$.
\end{lemma}

\begin{lemma} \label{theorem3}
	Assume that $F(\bm{z})$ is a log-concave function, $\bm{z} \in \mathbb{R}^{N}$ and that $\bm{z} = A\bm{y} + \bm{b}$ with $\bm{y} \in \mathbb{R}^{M}$, $A \in \mathbb{R}^{N \times M}$, $\bm{b} \in \mathbb{R}^{N}$. If $A$ has rank $N$, then ${g}(\bm{y}) = F(A\bm{y} + \bm{b})$ is also a log-concave function.
\end{lemma}
With Lemma \ref{theorem2} and Lemma \ref{theorem3}, Theorem~\ref{theorem0} follows. 

\subsection{Proof of Theorem~\ref{thm_random_search} and Corollary \ref{cor:compos}}
For this proof, we use the small $o$ notation, where $f(x)=o(g(x))$ as $x \rightarrow 0$ means that $\lim_{x\rightarrow 0} \frac{f(x)}{g(x)}=0$. The proof of Theorem~\ref{thm_random_search} is given by the following set of equations:
\begin{equation}
\begin{aligned}
\nabla f(\bm{x})^{\top} \bm{p} & = \nabla f(\bm{x})^{\top} \frac{f(\bm{x} + \Delta \cdot \bm{e}) - f(\bm{x})}{\Delta} \bm{e} \\
& \stackrel{(a)}{=}  \nabla f(\bm{x})^{\top}  \frac{ \nabla f(\bm{x})^{\top} \Delta \cdot \bm{e} + o(\Vert\Delta \cdot \bm{e}\Vert)}{\Delta} \bm{e} \\
& = (\nabla f(\bm{x})^{\top} \bm{e})^2 + \frac{o(\Vert\Delta \cdot \bm{e}\Vert)}{\Delta}\nabla f(\bm{x})^{\top} \bm{e} \\
& = (\nabla f(\bm{x})^{\top} \bm{e} )^2 + \frac{o(\Vert\Delta \cdot \bm{e}\Vert)}{\Vert\Delta \cdot \bm{e}\Vert}\nabla f(\bm{x})^{\top} \bm{e},
\end{aligned}
\end{equation}
where $(a)$ follows from the first order Taylor expansion of a differentiable function.
Note that $\frac{o(\Vert\Delta \cdot \bm{e}\Vert)}{\Vert\Delta \cdot \bm{e}\Vert } \rightarrow 0$ as $\Vert\Delta \cdot \bm{e}\Vert = \Delta \rightarrow 0$, for every $\bm{e}$. Then with probability 1 we can take $\Delta$ small enough and ensure that $\nabla f(\bm{x})^{\top} \bm{p}  > 0$.
The proof of Corollary~\ref{cor:compos} follows in a similar fashion using the chain rule:
%
%
%
\begin{equation*}
\begin{aligned}
& \nabla_{\bm{y}} f(g(\bm{y}))^{\top} \bm{p} \\
= & (J_g(\bm{y})^{\top}\nabla f(\bm{x})|_{\bm{x} = g(\bm{y})} )^{\top}\bm{p} \\
= &   \nabla f(\bm{x})^{\top}J_g(\bm{y}) J_g(\bm{y})^{\top}\frac{f(\bm{x} + \Delta \bm{e}) - f(\bm{x})}{\Delta} \bm{e}\\
= &\text{tr} (  J_g(\bm{y}) J_g(\bm{y})^{\top}\frac{f(\bm{x} + \Delta \bm{e}) - f(\bm{x})}{\Delta} \bm{e}  \nabla f(\bm{x}) ^{\top}) \\
\overset{(a)}{\geq} & \lambda_{min}(J_g(\bm{y}) J_g(\bm{y})^{\top}) \text{tr} (\frac{f(\bm{x} + \Delta \bm{e}) - f(\bm{x})}{\Delta} \bm{e}  \nabla f(\bm{x})^{\top})  \\
\overset{(b)}= &  \lambda_{min}( J_g(\bm{y}) J_g(\bm{y})^{\top}) \text{tr} (\nabla f(\bm{x})^{\top}\frac{f(\bm{x} + \Delta \bm{e}) - f(\bm{x})}{\Delta} \bm{e}  )  \\
\overset{(c)}{\geq} & 0,
\end{aligned}
\end{equation*}
where $(a)$ is the direct result from \cite{FangEtAl1994} to bound trace of a symmetric matrix by its smallest eigenvalue $\lambda_{min}(\cdot)$, $(b)$ comes from the fact that trace is invariant of cyclic permutations, and inequality $(c)$ is again based on Theorem \ref{thm_random_search} and that $\lambda_{min}(J_g(\bm{y}) J_g(\bm{y})^{\top}) \geq 0$. More specifically, if $J_g(\bm{y}) J_g(\bm{y})^{\top} \succ 0$, then $\nabla_{\bm{y}} f(g(\bm{y}))^{\top} \bm{p}  >0$.


\subsection{Proof of Theorem~\ref{theorem0}}
We reformulate $g(\bm{Q})$ into:
\begin{subequations}\label{formulation2}
	\begin{align}
	& g(\bm{Q}) \\
	= & \Pr \{ \bm{0}  \leq  - \underline{\bm{V}} + \bm{R}\bm{P} + \bm{X}\bm{Q} + \bm{\epsilon} \leq  - \underline{\bm{V}} + \overline{\bm{V}} \} \\
	=  & \Pr \{ \bm{0}  \leq  \bm{Y}(\bm{Q}) + \bm{\epsilon} \leq  \bm{u} \}, \label{eqn:YQ} \\
	= & \Pr \{ -\bm{Y}(\bm{Q})  \leq   \bm{\epsilon} \leq  \bm{u} - \bm{Y}(\bm{Q}) \}, \label{eqn:YQ_2}
	\end{align}
\end{subequations}
where $\bm{Y}(\bm{Q}) = - \underline{\bm{V}} + \bm{R}\bm{P} + \bm{X}\bm{Q}$ and $- \underline{\bm{V}} + \overline{\bm{V}} = \bm{u}$.

{Let $\bm{z} = \bm{u} - \bm{Y}(\bm{Q}) $, let $c = \prod_{i}^{N}u_i$, where $\bm{u} = [u_1, \cdots, u_N]^{\top}$. Then \eqref{eqn:YQ_2} can be compactly written as:}
\begin{equation}\label{eqn:prob_x}
\Pr \{ \bm{z} - \bm{u}  \leq   \bm{\epsilon} \leq \bm{z} \} \overset{\Delta}{=} F(\bm{z}).
\end{equation}

{With Lemma \ref{theorem2}, 
	we know that $F(\bm{z})$ is log concave in $\bm{z}$. Since $\bm{z} = \bm{u} - \bm{Y}(\bm{Q})  =  \bm{A} \bm{Q} + \bm{b}$, where $\bm{A}$ is full rank and $\bm{b}$ is a constant, Lemma \ref{theorem3} implies $g(\bm{Q})$ is a log-concave function.} 

\end{document}